\documentclass[a4paper,12pt]{article}
\usepackage[utf8]{inputenc}

\usepackage{amsmath}
\usepackage{amssymb}
\usepackage{amsthm}
\usepackage{mathtools}

\usepackage{cite}

\usepackage{booktabs}

\usepackage{esint}

\usepackage{mathrsfs}
\usepackage{bbm}

\usepackage{accents}

\usepackage{color}

\makeatletter
\let\my@saved@original@eqref\eqref 
\renewcommand*{\eqref}[1]{
  \begingroup
    \let\normalfont\relax
    \my@saved@original@eqref{#1}
  \endgroup
}
\makeatother
\makeatletter
\renewcommand*\env@matrix[1][r]{\hskip -\arraycolsep
  \let\@ifnextchar\new@ifnextchar
  \array{*\c@MaxMatrixCols #1}}
\makeatother

\newcommand{\ord}[1]{\mathcal{O}\left(#1\right)}
\newcommand{\e}{\mathrm{e}}
\newcommand{\pt}{\partial}
\newcommand{\jump}[1]{[\hspace*{-2pt}[#1]\hspace*{-2pt}]}
\newcommand{\eps}{\varepsilon}
\newcommand{\norm}[2]{\|{#1}\|_{#2}}
\newcommand{\bignorm}[2]{\left\| {#1} \right\|  _{#2}}
\newcommand{\tnorm}[1]{\left|\!\!\;\left|\!\!\;\left| {#1}
                       \right|\!\!\;\right|\!\!\;\right|}
\newcommand{\N}{\mathbb{N}}

\newcommand{\rarrow}{\quad\Rightarrow\quad}
\newcommand{\pmtrx}[1]{\ensuremath{\begin{pmatrix}#1 \end{pmatrix}}}
\newcommand{\mA}{\boldsymbol{A}}
\newcommand{\ds}{\,\mathrm{d}s}
\newcommand{\dt}{\,\mathrm{d}t}
\newcommand{\dx}{\,\mathrm{d}x}

\usepackage{multirow}
\title{Numerical analysis of a singularly perturbed 4th order problem with a shift term}
\author{Sebastian Franz\footnote{
          Institute of Scientific Computing, Technische Universit\"at Dresden, Germany.
          \mbox{e-mail}: sebastian.franz@tu-dresden.de},\and
        Kleio Liotati\footnote{
          Technische Universit\"at Dresden, Germany.
          \mbox{e-mail}: klio9965@gmail.com}
       }
\date{\today}

\usepackage{tikz}
\usepackage{pgfplots}
\pgfplotsset{compat=1.9}
\usetikzlibrary{calc}
\usetikzlibrary{external}
\tikzsetfigurename{figure_}
\tikzset{external/system call={pdflatex \tikzexternalcheckshellescape 
-interaction=batchmode -jobname "\image" "\texsource";
convert -density 600 -transparent white "\image.pdf" "\image.png"}}

\usepackage{fancyhdr} 
\newcommand{\scp}[1]{\langle #1 \rangle}

\newcommand{\PS}{\mathbbm{P}}
\newcommand{\U}{\mathcal{U}}
\newcommand{\V}{\mathcal{U}_h}

\theoremstyle{plain}
\newtheorem{thm}{Theorem}[section]
\newtheorem{lem}[thm]{Lemma}

\newtheorem{rem}[thm]{Remark}

\usepackage{color}

\begin{document}
  \pagestyle{fancy}
  \maketitle
  \begin{abstract}
    We consider a one-dimensional singularly perturbed 4th order problem with 
    the additional feature of a shift term. An expansion into a smooth term, boundary layers
    and an inner layer yields a formal solution decomposition, and together with a stability result 
    we have estimates for the subsequent numerical analysis. With classical layer adapted 
    meshes we present a numerical method, that achieves supercloseness and optimal convergence orders in 
    the associated energy norm. We also consider coarser meshes in view of the weak layers.
    Some numerical examples conclude the paper and support the theory.
  \end{abstract}


     
  \section{Introduction}
  In this paper we consider for $m\in\{1,\,2\}$ the singularly perturbed 4th order problem
  \begin{subequations}\label{eq:problem}
  \begin{align}
    Lu(x):=\eps^2 u^{(4)}(x)-b(x)u''(x)+c(x)u(x)+d(x)u(x-1)&=f(x),\,x\in\Omega=(0,2),\\
    u(x)&=\Phi(x),\,x\in(-1,0),\\
    u(0)=u(2)&=0,\\
    u^{(m)}(0)=u^{(m)}(2)&=0,\label{eq:problem:BC}
  \end{align}
  \end{subequations}
  where $\Phi$ is a given function with $\Phi(0)=u(0)=0$ and $\Phi^{(m)}(0)=u^{(m)}(0)=0$, which is not a practical restriction,
  and $b,c,d$ are smooth functions with $b\geq \beta ^2>0$ and $c-\frac{\norm{d}{L^\infty(1,2)}}{2} -\frac{\norm{b'}{L^\infty}^2}{2\beta ^2}\geq \delta>0$.
  The case $m=1$ in \eqref{eq:problem:BC} can be seen as modelling a clamped 1d beam, while 
  $m=2$ models a supported 1d beam.

  For second-order singularly perturbed problems with a shift (sometimes also called a delay) we find some papers on numerical analysis in the literature, see e.g. \cite{KumarKadalbajoo, Gupta, Bansal, Chakravarthy, KumarKumari,BFrLR22a,NX13}
  for the reaction-diffusion case and \cite{SR12,SR13,RS15,BFrLR22b} for the convection-diffusion case.
  For the fourth order problem this is the first paper on problems with a shift. 
  
  We will follow the classical way of analysing numerical methods for singularly perturbed problems, see also \cite{RST08}.
  To do this, we first provide a solution decomposition of our problem. One way to do this is to make assumptions about the signs of the coefficients
  and derive a maximum principle. We proceed in another way, without restricting the coefficients, and use a stability result
  involving Green's function estimates, see Section~\ref{sec:decomposition} and the appendix. Once the structure of the solution is known, the 
  construction of layer-adapted meshes is straightforward, see Section~\ref{sec:method}. But we will also look at 
  another construction, a more problem-orientated one in Section~\ref{sec:coarse}. Section~\ref{sec:numana} contains the 
  numerical analysis on the standard layer-adapted meshes.
  Finally, there are numerical examples in Section~\ref{sec:numerics}, which provide some numbers to support the theoretical claims.

  Notation: We will denote by $L^p(D)$ the classical Lebesgue norm over the domain $D$ and skip the reference to the domain when $D=\Omega$.

\section{Solution decomposition}\label{sec:decomposition}
  We give a derivation of a formal solution decomposition for the solution $u$ of \eqref{eq:problem} in the case of constant
  coefficients $b$ and $c$, but assume that it also holds for the case of variable coefficients. In the following we set $m=1$, but the derivation
  can easily be adapted to the other case.
  
  As a first step, we rewrite our problem \eqref{eq:problem} by splitting $\Omega$ at $x=1$. 
  For $u=u_1\chi_{[0,1)}+u_2\chi_{[1,2]}$ we get
  \begin{subequations}\label{eq:4thsystem}
    \begin{align}
      \eps^2 u_1^{(4)}(x)-b(x)u_1''(x)+c(x)u_1(x)&=f(x)-d(x)\Phi(x-1),\,x\in(0,1),\label{eq:4thsystem:1}\\
      \eps^2 u_2^{(4)}(x)-b(x)u_2''(x)+c(x)u_2(x)&=f(x)-d(x)u_1(x-1),\,x\in(1,2),\label{eq:4thsystem:2}\\
      u_1(0)=u_2(2)&=0,\\
      u_1'(0)=u_2'(2)&=0,\\
      \jump{u}(1)    =
      \jump{u'}(1)&  =
      \jump{u''}(1)  =
      \jump{u'''}(1) =0,\label{eq:4thsystem:continuity}
    \end{align}
  \end{subequations}
  where $\jump{v}(1):=v_2(1^+)-v_1(1^-)$ denotes the jump of $v$ at $x=1$.
  In addition to the original conditions, we have the continuity conditions \eqref{eq:4thsystem:continuity}.
  Together with the differential equations \eqref{eq:4thsystem:1} and \eqref{eq:4thsystem:2} they
  give $u\in C^4(\Omega)$ if $f\in C(\Omega)$.
  
  To derive the formal decomposition, we need some auxiliary problems. For the so-called outer expansion,
  we replace $u$ by 
  \[
    \sum_{k=0}^\infty \eps^k S_k(x):=\sum_{k=0}^\infty \eps^k (S_{k,left}(x)\chi_{(0,1)}(x)+S_{k,right}(x)\chi_{(1,2)}(x))
  \]
  in \eqref{eq:4thsystem}. Comparing the lowest order power of $\eps$ we get for $S_0$ from the system 
  \begin{align*}
    -bS_{0,left}''(x) +cS_{0,left}(x) &=f(x)-d\Phi(x-1),\,x\in\Omega=(0,1),\\
    -bS_{0,right}''(x)+cS_{0,right}(x)&=f(x)-dS_{0,left}(x-1),\,x\in\Omega=(1,2),\\
    S_{0,left}(0)  =S_{0,right}(2)&=0,\\
    \jump{S_0}(1) =0,\,
    \jump{S_0'}(1)&=0.
  \end{align*}
  Its solution holds $S_0\in C^2(\Omega)\setminus C^3(\Omega)$ if $f\in C(\Omega)$
  and it does not satisfy the second set of boundary conditions.
  So we use inner expansions to correct the boundary values and the regularity problem. For the 
  left boundary we use a scaled variable $\xi=x/\eps$, replace $u$ in \eqref{eq:4thsystem} by 
  \[
    \sum_{k=1}^\infty\eps^k E_{k,left}(x):=\sum_{k=1}^\infty \eps^k \tilde E_{k,left}(\xi)
  \]
  and include the boundary condition correction. At the lowest level of $\eps$ this is
  \begin{align*}
    \tilde E_{1,left}^{(4)}(\xi)-b\tilde E_{1,left}''(\xi)&=0,\\
    \tilde E_{1,left}'(0)&=-S_0'(0).
  \end{align*}
  To get a unique solution, we assume that $\tilde E_{1,left}$ is exponentially decaying. 
  As a result, the boundary condition for the derivative is satisfied, but we introduce 
  a discrepancy in the first boundary condition:
  \[
    (S_0+\eps E_{1,left})(0)=\eps E_{1,left}(0)=\ord{\eps}.
  \]
  We will deal with this in the next step of our expansion.
  Using $\eta=(2-x)/\eps$ and $E_{k,right}(x)=\tilde E_{k,right}(\eta)$ we can apply the same idea to correct the boundary value
  at $x=2$.
  
  To resolve the regularity issue at $x=1$, we introduce two inner expansions. With the two scaled variables
  $\psi=(1-x)/\eps=\eta-1/\eps$ and $\theta=(x-1)/\eps=\xi-1/\eps$ we replace $u$ in \eqref{eq:4thsystem} by
  \begin{align*}
    \sum_{k=3}^\infty \eps^k W_k(x)
    =\sum_{k=3}^\infty \eps^k (\tilde W_{k,left}(\psi)\chi_{(0,1)}(x)+\tilde W_{k,right}(\theta)\chi_{(1,2)}(x)).
  \end{align*}
  The lowest order of $\eps$ now gives the following coupled problem
  \begin{align*}
    \tilde W_{3,left}^{(4)}(\psi)   -b\tilde W_{3,left}''(\psi)&=-d\tilde E_{1,right}(\psi),\\
    \tilde W_{3,right}^{(4)}(\theta)-b\tilde W_{3,right}''(\theta)&=-d\tilde E_{1,left}(\theta),\\
    \tilde W_{3,left}''(0)&=\tilde W_{3,right}''(0),\\
    -\tilde W_{3,left}'''(0)&=\tilde W_{3,right}'''(0)-\jump{S_0'''}(1),
  \end{align*} 
  where we have included the correction in the second continuity condition. We have also included 
  as right-hand sides the shifted terms resulting from the boundary corrections, which were not treated before.
  Note that $W_3$ and its first derivative are not continuous at $x=1$.
  Let 
  \[
    V_0:=S_0+\eps(E_{1,left}+E_{1,right})+\eps^3 W_3
  \]
  be the sum of the components derived so far. Then, we have
  \begin{align*}
    \jump{V_0''}(1)=\jump{V_0'''}(1)&=0,\,
    \jump{V_0}(1)=\eps^3\jump{W_3}(1)=\ord{\eps^3}, \,
    \jump{V_0'}(1)=\eps^3\jump{W_3'}(1)=\ord{\eps^2},\\
    V_0(0) &=\eps E_{1,left}(0)+\ord{\e^{-\frac{\beta }{\eps}}},\,
    V_0'(0) =\ord{\e^{-\frac{\beta }{\eps}}},\\
    V_0(2) &=\eps E_{1,right}(2)+\ord{\e^{-\frac{\beta }{\eps}}},\,
    V_0'(2) =\ord{\e^{-\frac{\beta }{\eps}}}.
  \end{align*}
  The remaining jumps and $\ord{\eps}$-violations of the boundary conditions can all be incorporated
  into the $S$-problems of the next steps in the expansion. In addition, let us look at the residual, the terms
  so far not included into our differential problems. They are 
  \[
    \eps^2 S_0^{(4)}+\eps c (E_{1,left}+E_{1,right})+\eps^3 c W_3+\eps^3 d W_{3,left}(\cdot -1).
  \]
  The first can be included as the right-hand side in the problem for $\eps^2 S_2$, the next two in the problems for $\eps^3 E_3$ 
  and the third one in the problem for $\eps^5 W_5$. Note that these are the components of $V_2$ with
  \[
    V_k:=\eps^kS_k+\eps^{k+1}(E_{k+1,left}+E_{k+1,right})+\eps^{k+3}W_{k+3}.
  \]
  The last term with the shift-operator can be included in the right-hand side of the problem for $\eps^5 E_{5,right}$.
  
  In general, we have the following sub-problems. For $k\geq 0$ the outer expansion solves
  \begin{subequations}\label{eq:outerEx}
  \begin{align}
      -bS_{k,left}''(x) +cS_{k,left}(x) &=\begin{cases}
                                            f(x)-d\Phi(x-1),&k=0,\\
                                            0,& k=1,\\
                                            -S_{k-2,left}^{(4)},&k\geq 2,
                                          \end{cases} & x\in(0,1),\\
      -bS_{k,right}''(x)+cS_{k,right}(x)&=-dS_{k,left}(x-1)+\begin{cases}
                                            f(x),&k=0,\\
                                            0,&k=1,\\
                                            -S_{k-2,right}^{(4)},&k\geq 2,
                                          \end{cases} & x\in(1,2),
    \end{align}\vspace{-1em}
    \begin{align}
      S_{k,left}(0)  &=\begin{cases}
                          0, & k=0,\\
                        -E_{k,left}(0),& k\geq 1,
                      \end{cases}\quad
      S_{k,right}(2)  =\begin{cases}
                          0, & k=0,\\
                        -E_{k,right}(2),& k\geq 1,
                      \end{cases}\\
      \jump{S_k}(1)&=\begin{cases}
                        0, & k\leq 2,\\
                        -\jump{W_k}(1), & k\geq 3,
                      \end{cases}\quad
      \jump{S_k'}(1) =\begin{cases}
                        0, & k\leq 1,\\
                        -\jump{W_{k+1}'}(1),&k\geq 2.
                      \end{cases}
  \end{align}
  \end{subequations}
  The inner expansion for the left boundary layer solves for $\xi\in[0,\infty)$ and $k\geq 1$ 
  \begin{subequations}\label{eq:innerExE1}
  \begin{align}
    \tilde E_{k,left}^{(4)}(\xi)-b\tilde E_{k,left}''(\xi)&=\begin{cases}
                                                              0,&k\leq 2,\\
                                                              -c\tilde E_{k-2,left}(\xi),&k\geq 3,
                                                            \end{cases}\\
    \tilde E_{k,left}'(0)&=-S_k'(0),\,\text{exp. decay},
  \end{align}
  \end{subequations}
  while for the right boundary layer we have for $\eta\in[0,\infty)$ and $k\geq 1$
  \begin{subequations}\label{eq:innerExE2}
  \begin{align}
    \tilde E_{k,right}^{(4)}(\eta)-b\tilde E_{k,right}''(\eta)&=\begin{cases}
                                                                  0,&k\leq 2,\\
                                                                  -c\tilde E_{k-2,right}(\eta),&k\in\{3,4\},\\
                                                                  -c\tilde E_{k-2,right}(\eta)-d\tilde W_{k-2,left}(\eta),&k\geq 5,
                                                                \end{cases}\\
    \tilde E_{k,right}'(0)&=-S_k'(2),\,\text{exp. decay}.
  \end{align}
  \end{subequations}
  Finally, the inner expansion of the inner layer solves for $\psi,\theta\in[0,\infty)$ and $k\geq 3$
  \begin{subequations}\label{eq:innerExW}
  \begin{align}
    \tilde W_{k,left}^{(4)}(\psi)   -b\tilde W_{k,left}''(\psi)&=-d\tilde E_{k-2,right}(\psi)+\begin{cases}
                                                                                                0,&k\leq 4,\\
                                                                                                -c\tilde W_{k-2,left}(\psi),&k\geq 5,
                                                                                              \end{cases}\\
    \tilde W_{k,right}^{(4)}(\theta)-b\tilde W_{k,right}''(\theta)&=-d\tilde E_{k-2,left}(\theta)+\begin{cases}
                                                                                                    0,&k\leq 4,\\
                                                                                                    -c\tilde W_{k-2,right}(\theta),&k\geq 5,
                                                                                                  \end{cases}\\
    \tilde W_{k,left}''(0)&=\tilde W_{k,right}''(0),\\
    -\tilde W_{k,left}'''(0)&=\tilde W_{k,right}'''(0)-\jump{S_{k-3}'''}(1),\,\text{exp. decay}.
  \end{align}
  \end{subequations}
  With these problems we can define $V_k$ uniquely for any $k\geq 0$ and by above considerations, we have
  the formal solution decomposition 
  \begin{align*}
    u(x)&=\sum_{k=0}^\infty V_k+\ord{\e^{-\frac{\beta }{\eps}}}\\
        &= \sum_{k=0}^\infty \eps^k S_k(x)
          +\sum_{k=1}^\infty \eps^k (E_{k,left}(x)+E_{k,right}(x))
          +\sum_{k=3}^\infty \eps^k W_k(x)+\ord{\e^{-\frac{\beta }{\eps}}}.
  \end{align*}
  In the case $m=2$ the inner expansions of $E_{left}$ and $E_{right}$ start with $k=2$, but otherwise the same derivation holds.
  \begin{lem}\label{lem:stability}
    Let us consider the system \eqref{eq:4thsystem:1} and \eqref{eq:4thsystem:2} together with the more general conditions
    \begin{subequations}\label{eq:problem:BC2}
    \begin{align}
      u_1(0)=\alpha_1,\,
      u_1(1)&=\alpha_2,\,
      u_2(1)=\alpha_2+\delta_1,\,
      u_2(2)=\alpha_3,\\
      u_1'(0)=\beta_1,\,
      u_1'(1)&=\beta_2,\,
      u_2'(1)=\beta_2+\delta_2,\,
      u_2'(2)=\beta_3,
    \end{align}
    \end{subequations}
    where $\alpha_1,\,\alpha_3,\,\beta_1,\,\beta_3,\,\delta_1,\,\delta_2$ are given parameters and
    $\alpha_2,\,\beta_2$ are chosen, such that 
    \begin{equation}
      \jump{u''}(1)=\jump{u'''}(1)=0.\label{eq:cont}
    \end{equation}
    Then it follows that
    \[
      \norm{u}{L^\infty(0,2)}\lesssim\norm{f}{L^\infty(0,2)}+|\alpha_1|+|\alpha_3|+|\delta_1|+|\delta_2|+|\beta_1|+|\beta_3|.
    \]
  \end{lem}
  \begin{proof}
    The proof is rather technical and is deferred to the appendix.
  \end{proof}
  
  \begin{thm}\label{thm:soldec}
    The solution $u$ of problem~\eqref{eq:problem} can be written for $x\in\Omega$ as
    \[
      u(x)=S(x)+E_1(x)+E_2(x)+W_1(x)+W_2(x),
    \]
    where we have for $0\leq k\leq q+1$
    \begin{align*}
      \norm{S^{(k)}}{L^2(0,1)}+\norm{S^{(k)}}{L^2(1,2)} & \lesssim 1,\\
      |E_1^{(k)}(x)|&\lesssim \eps^{m-k}\e^{-\frac{\beta x}{\eps}},&
      |E_2^{(k)}(x)|&\lesssim \eps^{m-k}\e^{-\frac{\beta (2-x)}{\eps}},\\
      |W_1^{(k)}(x)\chi_{(0,1)}(x)|&\lesssim \eps^{3-k}\e^{-\frac{\beta (1-x)}{\eps}},&
      |W_2^{(k)}(x)\chi_{(1,2)}(x)|&\lesssim \eps^{3-k}\e^{-\frac{\beta (x-1)}{\eps}}.
    \end{align*}
  \end{thm}
  \begin{proof}
    Using the decomposition
    \[
      u(x)=\sum_{k=0}^\ell V_\ell(x)+R(x)
    \]
    with the remainder $R$, we can consider the problem for $u-R$. Here we have
    \begin{align*}
      |(u-R)(0)|+|(u-R)(2)|  &\lesssim \eps^{\ell+1}\\
      |(u-R)'(0)|+|(u-R)'(2)|&\lesssim \e^{-\sqrt{b}/\eps}\\
      |\jump{(u-R)''}(1)|    &\lesssim \eps^{\ell+3}\\
      |\jump{(u-R)'''}(1)|   &\lesssim \eps^{\ell+2}\\
      \norm{L(u-R)}{L^\infty(0,2)} &\lesssim \eps^{\ell+1}.
    \end{align*}
    Lemma~\ref{lem:stability} gives $\norm{u-R}{L^\infty(0,2)}\lesssim \eps^{\ell+1}$
    and the remainder can be incorporated into $S$ together with the contributions $S_0$ to $S_\ell$. 
    Similarly we combine the different layer components into $E_1$, $E_2$, $W_1$ and $W_2$ and the proof 
    is done by choosing $\ell$ large enough.
  \end{proof}

  We will also denote the boundary layers by $E=E_1+E_2$ and the inner layer by $W=W_1\chi_{(0,1)}+W_2\chi_{(1,2)}$.
  A visual confirmation is given for two examples in Figures~\ref{fig:plot} and \ref{fig:plot2} in Section~\ref{sec:numerics}.

\section{Numerical method}\label{sec:method}

  To derive our numerical method, we introduce the auxiliary variable 
  \[
    w=\eps u''.
  \]
  Note that by Theorem~\ref{thm:soldec} we immediately have a decomposition for $w$:
  \[
    w=\tilde S+\tilde E_1+\tilde E_2+\tilde W_1+\tilde W_2,
  \]
  where for $0\leq k\leq q+1$ we have 
  \begin{align*}
    |\tilde S^{(k)}(x)|&\lesssim\eps,\\
    |\tilde E_1^{(k)}(x)|&\lesssim \eps^{m-1-k}\e^{-\frac{\beta x}{\eps}},&
    |\tilde E_2^{(k)}(x)|&\lesssim \eps^{m-1-k}\e^{-\frac{\beta (2-x)}{\eps}},\\
    |\tilde W_1^{(k)}(x)\chi_{(0,1)}(x)|&\lesssim \eps^{2-k}\e^{-\frac{\beta (1-x)}{\eps}},&
    |\tilde W_2^{(k)}(x)\chi_{(1,2)}(x)|&\lesssim \eps^{2-k}\e^{-\frac{\beta (x-1)}{\eps}}.
  \end{align*}
  Similar to the previous notation we set 
  $\tilde E=\tilde E_1+\tilde E_2$ and $\tilde W=\tilde W_1\chi_{(0,1)}+\tilde W_2\chi_{(1,2)}$.
  
  Now we rewrite our problem~\eqref{eq:problem} into its variational formulation. Let
  \[
    \U:=\begin{cases}
         H_0^1(\Omega)\times H^1(\Omega),& m=1,\\
         H_0^1(\Omega)\times H_0^1(\Omega), &m=2.
       \end{cases}
  \]
  Then the problem reads: Find $(u,w)\in \U$ such that 
  for all $(y,z)\in\U$ it holds
  \begin{align}
    B((u,w),(y,z))
      &:=-\eps\scp{w',y'}+\scp{bu',y'}+\scp{b'u',y}+\scp{cu,y}+\scp{du(\cdot-1),y}_{(1,2)}\notag\\
         &\hspace*{2cm}+\eps\scp{u',z'}+\scp{w,z}\notag\\
      &= \scp{f,y}-\scp{d\Phi(\cdot-1),y}_{(0,1)}=:F(y).\label{eq:varprob}
  \end{align}
  Note that for $m=1$ the normal-boundary condition is only weakly enforced.
  
  For the bilinear form defined in \eqref{eq:varprob} we have coercivity
  \begin{align}
    B((u,w),(u,w))
      &:=-\eps\scp{w',u'}+\scp{bu',u'}+\scp{b'u',u}+\scp{cu,u}+\scp{du(\cdot-1),u}_{(1,2)}\notag\\
           &\hspace*{2cm}+\eps\scp{u',w'}+\scp{w,w}\notag\\
      &=\scp{bu',u'}+\scp{b'u',u}+\scp{cu,u}+\scp{du(\cdot-1),u}_{(1,2)}+\scp{w,w}\notag\\
      &\geq \norm{w}{L^2}^2+\frac{\beta ^2}{2}\norm{u'}{L^2}^2+\left( c-\frac{\norm{d}{L^\infty(1,2)}}{2} -\frac{\norm{b'}{L^\infty}^2}{2\beta ^2}\right)\norm{u}{L^2}^2\notag\\
      &\geq \norm{w}{L^2}^2+\frac{\beta ^2}{2}\norm{u'}{L^2}^2+\delta\norm{u}{L^2}^2=:\tnorm{(u,w)}^2.\label{eq:coercivity}
  \end{align}
  \begin{rem}\label{rem:weak}
    Note that this norm is weak, i.e. contributions of the layers vanish for $\eps\to 0$:
    \[
      \tnorm{(S,\tilde S)}\lesssim 1,
      \quad
      \tnorm{(E,\tilde E)}\lesssim\eps^{m-1/2}\to 0,
      \quad
      \tnorm{(W,\tilde W)}\lesssim\eps^{5/2}\to 0.
    \]
    Using weighted norms as in \cite{MS21} only reduces the powers of $\eps$, but the norm remains weak.
  \end{rem}
  
  Our numerical method will be given on a layer-adapted mesh. We will start with an S-type mesh, see~\cite{RL99}, 
  in Section~\ref{sec:numana} and consider variations of it in Section~\ref{sec:coarse}. 
  Let the number of mesh cells $N\in\N$ be divisible by 8 and 
  \[
    \lambda:=\min\left\{\frac{\sigma}{\beta }\eps\ln(N),\frac{1}{4}\right\}
  \]
  be the transition point, where we assume that $\eps$ is small enough such that $\lambda<\frac{1}{4}$.
  The parameter $\sigma>0$ is defined in the following numerical analysis. 
  The nodes of the mesh $\Omega_h$ are then given by
  \begin{gather}\label{eq:Smesh}
    x_i=\begin{cases}
           \frac{\sigma\eps}{\beta }\phi\left(\frac{4i}{N}\right),&       0\leq i \leq \frac{N}{8},\\
           \frac{4i}{N}(1-2\lambda)+2\lambda-\frac{1}{2},& \frac{N}{8}\leq i\leq \frac{3N}{8},\\
           1-\frac{\sigma\eps}{\beta }\phi\left(2-\frac{4i}{N}\right),&   \frac{3N}{8}\leq i \leq \frac{N}{2},\\
           1+x_{i-N/2},& \frac{N}{2}\leq i\leq N.
        \end{cases}
  \end{gather}
  
  We denote the so-called mesh-defining function by $\phi$, which is monotonically increasing with $\phi(0)=0$ and $\phi(1/2)=\ln(N)$,
  see \cite{RL99} for the exact conditions on $\phi$.
  In the numerical analysis we will also use the associated mesh characterising function $\psi=\e^\phi$, or more precisely
  the quantity $\max|\psi'|:=\max\limits_{t\in[0,1/2]}|\psi'(t)|$. Two of the most common S-type meshes are
  the Shishkin mesh with
  \[
    \phi(t) = 2t\ln N,\quad
    \psi(t) = N^{-2t},\quad
    \max|\psi'| \leq 2\ln N,\quad
    h= \frac{8\sigma\eps}{\beta } N^{-1}\ln N
  \]
  and the Bakhvalov S-mesh with
  \[
    \phi(t)=-\ln(1-2t(1-N^{-1})),\quad
    \psi(t)=1-2t(1-N^{-1}),\quad
    \max|\psi'|\leq 2,\quad 
    h\leq \frac{\sigma\eps}{\beta }\ln 5.
  \]
  For the mesh widths within the layers it holds, see \cite{RL99},
  \[
    h_i\lesssim \eps N^{-1}\max|\psi'|\e^{\frac{\beta  x}{\sigma\eps}},\,x\in[x_{i-1},x_i].
  \]
  
  The mesh $\Omega_h$ is then given by
  \[
    \Omega_h:=\{[x_{i-1},x_i],\,i\in\{1,\dots,N\}\}.
  \]
  Note that the maximum mesh width in the layer regions is $h\lesssim\eps$ and outside $H\lesssim N^{-1}$.

  Finally, we define our discrete space on the given mesh by
  \[
    \V:=\{(u_h,w_h)\in\U:u_h|_T\in\PS_q(T),\,w_h|_T\in\PS_q(T),\,\forall T\in\Omega_h\},
  \]
  where $\PS_q(T)$ denotes the space of polynomials of maximum degree $q$ over $T$.
  Our discrete method is given by: Find $(u_h,w_h)\in\V$ such that for all $(y,z)\in\V$ it holds
  \begin{gather}\label{eq:discr}
    B((u_h,w_h),(y,z))=F(y).
  \end{gather}
  
\section{Numerical analysis} \label{sec:numana}
  Note that by $\V\subset\U$, \eqref{eq:varprob} and \eqref{eq:discr} we have Galerkin orthogonality
  \begin{gather}\label{eq:Galerkin_orthogonality}
    B((u-u_h,w-w_h),(y,z))=0,\text{ for all }(y,z)\in\V.
  \end{gather}
  To analyse the error of the numerical method, we define a compound interpolation operator
  $I=(I_1,I_2):C(\Omega)^2\to\V$ locally for each $i\in\{1,\dots,N\}$ and $v\in C([x_{i-1},x_i])$ by
  \begin{align*}
    (I_jv-v)(x_{i-1})=0,\quad
    (I_jv-v)(x_{i})&=0,\\
    \int_{x_{i-1}}^{x_i}(I_jv-v)(x)\xi(x)\dx&=0,\,\forall \xi\in\PS_{q-2}([x_{i-1},x_i]),\,j\in\{1,2\}.\\
  \end{align*}
  Here $I_1$ and $I_2$ are defined similarly and for $m=2$ they are the same, but for $m=1$ the 
  spaces the operator maps into are different. Let us decompose the errors as follows
  \begin{align*}
    u-u_h=(u-I_1u)-(u_h-I_1u)=:\eta_u+\xi_u,\\
    w-w_h=(w-I_2w)-(w_h-I_2w)=:\eta_w+\xi_w,
  \end{align*}
  where $(\eta_u,\eta_w)$ is the interpolation error and $(\xi_u,\xi_w)\in\V$ is the discrete error. Then 
  with \eqref{eq:Galerkin_orthogonality} and the coercivity~\eqref{eq:coercivity} we obtain the error inequality
  \[
    \tnorm{(\xi_u,\xi_w)}^2
     \leq B((\xi_u,\xi_w),(\xi_u,\xi_w))
     = B((\eta_u,\eta_w),(\xi_u,\xi_w)).
  \]
  \begin{lem}\label{lem:error_estimation}
    It holds
    \[
      B((\eta_u,\eta_w),(\xi_u,\xi_w))
        \lesssim(\norm{\eta_u}{L^2}+\norm{\eta_w}{L^2}+\norm{b-\bar b}{L^\infty}\norm{\eta_u'}{L^2})\tnorm{(\xi_u,\xi_w)},
    \]
    where $\bar b$ is a cell-wise average of $b$.
  \end{lem}
  \begin{proof}
    Let us look at the individual terms of 
    \begin{align*}
      B((\eta_u,\eta_w),(\xi_u,\xi_w))
        &=-\eps\scp{\eta_w',\xi_u'}+\eps\scp{\eta_u',\xi_w'}+\scp{b\eta_u',\xi_u'}+\scp{b'\eta_u',\xi_u}+\scp{c\eta_u,\xi_u}+\scp{\eta_w,\xi_w}\notag\\
           &\hspace*{2cm}+\scp{d\eta_u(\cdot-1),\xi_u}_{(1,2)}\\
           &=:\sum_{k=1}^7 B_k
    \end{align*}
    on each cell $T_i:=[x_{i-1},x_i]$.
    By definition of the interpolation operator $I$ and integration by parts we obtain
    \begin{align*}
      \scp{\eta_w',\xi_u'}_{T_i}
        &=\eta_w(x)\xi_u'(x)|_{x_{i-1}}^{x_i}-\int_{x_{i-1}}^{x_i}\eta_w(x)\xi_u''(x)\dx=0,\\
      \scp{\eta_u',\xi_w'}_{T_i}
        &=0
    \end{align*}
    and therefore, $B_1$ and $B_2$ are zero. For $B_3$ we introduce the cell-wise average
    \[
      \bar b|_{T_i}=\int_{T_i}b(x)\dx.
    \]
    Then it holds again with the interpolation property
    \begin{align*}
      \scp{b\eta_u',\xi_u'}
        &=\scp{(b-\bar b)\eta_u',\xi_u'}+\bar b\scp{\eta_u',\xi_u'}
         \lesssim\norm{b-\bar b}{L^\infty(T_i)}\norm{\eta_u'}{L^2(T_i)}\norm{\xi_u'}{L^2(T_i)}.
    \end{align*}
    For $B_4$  we also apply integration by parts and have
    \begin{align*}
      \scp{b'\eta_u',\xi_u}_{T_i}
        &=\eta_u(x)b'(x)\xi_u(x)|_{x_{i-1}}^{x_i}-\scp{b''\eta_u,\xi_u}_{T_i}-\scp{b'\eta_u,\xi_u'}_{T_i}\\
        &\lesssim\norm{\eta_u}{L^2(T_i)}(\norm{\xi_u}{L^2(T_i)}+\norm{\xi_u'}{L^2(T_i)}).
    \end{align*}
    The last three terms $B_5 $ to $B_7$ are simply estimated by a Cauchy-Schwarz inequality
    \begin{align*}
      \scp{c\eta_u,\xi_u}&\lesssim \norm{\eta_u}{L^2}\norm{\xi_u}{L^2},\\
      \scp{\eta_w,\xi_w} &\lesssim \norm{\eta_w}{L^2}\norm{\xi_w}{L^2},\\
      \scp{d\eta_u(\cdot-1),\xi_u}_{(1,2)}&\lesssim\norm{\eta_u}{L^2(0,1)}\norm{\xi_u}{L^2(1,2)}.
    \end{align*}
    Combining the local and global estimates completes the proof.
  \end{proof}
  The main ingredient of interpolation error estimates are local estimates. Here we have on a cell $\tau_i$ of width $h_i$
  \begin{gather}\label{eq:inter:local}
    \norm{(v-Iv)^{(\ell)}}{L^2(\tau_i)}
      \lesssim \norm{h_i^{s-\ell}v^{(s)}}{L^2(\tau_i)}
  \end{gather}
  for $0\leq\ell< s\leq q+1$. 
  \begin{lem}\label{lem:inter:Stype}
    It holds for $u,w\in H^{q+1}(\Omega)$ and $\sigma\geq q+1$ using the results of Section~\ref{sec:decomposition}
    \begin{align*}
      \norm{\eta_u}{L^2}  &\lesssim (h+N^{-1})^{q+1}     +\eps^{m}(N^{-1}\max|\psi'|)^{q+1},\\
      \norm{\eta_w}{L^2}  &\lesssim \eps (h+N^{-1})^{q+1}+\eps^{m-1}(N^{-1}\max|\psi'|)^{q+1},\\
      \norm{\eta_u'}{L^2} &\lesssim (h+N^{-1})^q         +\eps^{m-1/2}(N^{-1}\max|\psi'|)^q.
    \end{align*}    
  \end{lem}
  \begin{proof}
    Interpolation error estimation of $I$ on an S-type mesh is standard, see~\cite{RL99}. Nevertheless, 
    we present it here again to highlight the changes we make in Section~\ref{sec:coarse} on
    coarsened S-type meshes.
    
    With \eqref{eq:inter:local} for $S$ and $\tilde S$ we obtain the estimates
    \begin{align*}
      \norm{S-I_1S}{L^2}&\lesssim(h+N^{-1})^{q+1},&
      \norm{\tilde S-I_2\tilde S}{L^2}&\lesssim\eps(h+N^{-1})^{q+1},\\
      \norm{(S-I_1S)'}{L^2}&\lesssim(h+N^{-1})^q.
    \end{align*}
    For the boundary layer $E_1$ (and similarly for $E_2$) we use additionally the $L^\infty$-stability of $I$ and the choice of the 
    transition point $\lambda$ of the mesh. We get
    \begin{align*}
      \norm{E_1-I_1E_1}{L^2}^2
        &= \norm{E_1-I_1E_1}{L^2(0,\lambda)}^2+\norm{E_1-I_1E_1}{L^2(\lambda,2)}^2\\
        &\lesssim\sum_{\tau_i\in\Omega_h([0,\lambda])}(\eps N^{-1}\max|\psi'|)^{2(q+1)}\eps^{2(m-(q+1))}\norm{\e^{(\frac{q+1}{\sigma}-1)\frac{x\beta }{\eps}}}{L^2(\tau_i)}^2+\norm{E_1}{L^\infty(\lambda,2)}^2\\
        &\lesssim \eps^{2m}((N^{-1}\max|\psi'|)^{q+1}+N^{-2\sigma})
         \lesssim \eps^{2m}(N^{-1}\max|\psi'|)^{q+1}
    \end{align*}
    due to $\sigma\geq q+1$. 
    Note, that for $\sigma>q+1$ we could gain another half power of $\eps$.
    For $E_2$ we get the same result and for $W$ we have to replace $m$ with $2$.
    Similarly we have
    \begin{align*}
      \norm{\tilde E-I_2\tilde E}{L^2}&\lesssim \eps^{m-1}(N^{-1}\max|\psi'|)^{q+1},&
      \norm{\tilde W-I_2\tilde W}{L^2((0,1)\cup(1,2))}&\lesssim \eps^2(N^{-1}\max|\psi'|)^{q+1}.
    \end{align*}
    These estimates give the first two results of the lemma. To estimate the derivative of the interpolation error,
    we demonstrate the procedure only for $E_1$, the others follow similarly. It holds
    \begin{align*}
      \norm{(E_1-I_1 E_1)'}{L^2(0,\lambda)}^2
        &\lesssim \sum_{\tau_i\in\Omega_h([0,\lambda])}(\eps N^{-1}\max|\psi'|)^{2q}\eps^{2(m-(q+1))}\norm{\e^{(\frac{q}{\sigma}-1)\frac{x\beta }{\eps}}}{L^2(\tau_i)}^2\\
        &\lesssim \eps^{2m-1}(N^{-1}\max|\psi'|)^{2q},\\
      \norm{(E_1-I_1 E_1)'}{L^2(\lambda,1-\lambda)}
        &\lesssim \norm{E_1'}{L^2(\lambda,1-\lambda)}+N\norm{I_1 E_1}{L^2(\lambda,1-\lambda)}\\
        &\lesssim \eps^{m-1/2}N^{-\sigma}+N^{1-\sigma}\eps^{m}.
    \end{align*}
    In the remaining parts of $\Omega$ we use one of the two techniques above, depending on the local mesh. The other layer parts can be estimated similarly 
    and the proof is complete.
  \end{proof}
  Combining the previous lemmas yields the convergence result in
  the energy norm.
  \begin{thm}[Supercloseness and convergence]\label{thm:conv:Smesh}
    For the solutions $(u,w)\in\U$ of \eqref{eq:varprob} and $(u_h,w_h)\in\V$ of \eqref{eq:discr} it holds under the assumptions
    $u,w\in H^{q+1}(\Omega)$ and $\sigma\geq q+1$, the supercloseness result
    \[
      \tnorm{(I_1u-u_h,I_2w-w_h)}\lesssim(h+N^{-1}\max|\psi'|)^{q+1}
    \]
    and the convergence result
    \[
      \tnorm{(u-u_h,w-w_h)}\lesssim(h+N^{-1}\max|\psi'|)^q.
    \]
  \end{thm}
  \begin{proof}
    Substituting the interpolation error estimates into Lemma~\ref{lem:error_estimation} gives the supercloseness
    result. This, together with a triangle inequality and the interpolation error again, gives the convergence
    result.
  \end{proof}
  \begin{rem}\label{rem:super}
    Since the discrete error converges with a higher order than the error itself, we have a supercloseness phenomenon.
    We can exploit this by applying a post-processing with an interpolation on a macro mesh of the discrete solution into a polynomial space
    of one degree higher, see e.g. \cite{ST03,Fr12} for a 2d setting of this approach. 
    The post-processed solution then converges with order $q+1$ (assuming $\sigma\geq q+2$ and regularity of the solution).
    
    We also obtain optimal error convergence in the $L^2$-norm
    \begin{align*}
      \norm{u-u_h}{L^2}+\norm{w-w_h}{L^2}
      &\leq \norm{\eta_u}{L^2}+\norm{\xi_u}{L^2}+\norm{\eta_w}{L^2}+\norm{\xi_w}{L^2}\\
      &\lesssim(h+N^{-1}\max|\psi'|)^{q+1}.
    \end{align*}
  \end{rem}
  
\section{Coarser meshes}\label{sec:coarse}
  According to our solution decomposition the inner layers $W$ and $\tilde W$ are weak layers
  and we will modify the meshes near $x=1$. Note that $E$ and $\tilde E$ are also weak layers 
  and similar approaches can be used for them near the boundary. 
  
  \subsection{Weak equidistant mesh}
  In \cite{BFrLR22b} the solution also contained a weak layer. This allowed a modification of the mesh 
  in order to coarsen the layer part of the mesh. We will apply this mesh idea here too, adapted to the even weaker inner layer.
  As mentioned above, this could also be done for the boundary layers.
  
  Let us define a second transition point in addition to the transition point $\lambda$ 
  \[
    \mu:=\begin{cases}
           \frac{1}{4},&q\leq 2,\\
           \min\left\{\frac{\eps^{1-\frac{5}{2(q+1)}}}{\beta },\frac{1}{4}\right\},&q>2.
         \end{cases}
  \]
  For $q>2$ we assume $\eps$ to be small enough such that $\mu<\frac{1}{4}$.
  We then use an S-type mesh with $N/8$ cells in each $(0,\lambda)$ and $(2-\lambda,2)$ as before, and choose equidistant meshes
  in $(\lambda,1-\mu)$, $(1-\mu,1+\mu)$ and $(1+\mu,2-\lambda)$ with $N/4$ cells each. Let $\tilde h$ denote the 
  mesh width in $(1-\mu,1+\mu)$. For it holds
  \[
    \tilde h\lesssim \begin{cases}
                       N^{-1},&q\leq 2,\\
                       N^{-1}\eps^{1-\frac{5}{2(q+1)}},&q>2
                     \end{cases}
    \rarrow
    \tilde h\lesssim N^{-1}.
  \]
  According to the previous analysis, we only need to estimate the interpolation error on this mesh in order to obtain the approximation error and the convergence error.

  \begin{lem}\label{lem:inter:weakeq}
    Assuming $\e^{-\eps^{-1/q}}\leq(h+N^{-1})^{q-2}$, we have 
    on a mesh consisting of an S-type mesh with $\sigma\geq q+1$ for the boundary layers and an equidistant weak mesh
    for the interpolation operator $I$
    \begin{align*}
      \norm{u-I_1 u}{L^2}&\lesssim (h+N^{-1})^{q+1}+\eps(N^{-1}\max|\psi'|)^{q+1},\\
      \norm{w-I_2 w}{L^2}&\lesssim (h+N^{-1})^{q+1/2},\\
      \norm{(u-I_1 u)'}{L^2}&\lesssim (h+N^{-1})^q+\eps^{1/2}(N^{-1}\max|\psi'|)^q.
    \end{align*}
  \end{lem}
  \begin{proof}
    The estimates for $S,\, \tilde S,\, E$ and $\tilde E$ are as in Lemma~\ref{lem:inter:Stype}, also for the derivatives.
    So we only need to look at the inner layers. 
    For $\tilde W$ we distinguish between the polynomial degrees. For $q=1$ we estimate directly
    with \eqref{eq:inter:local} and $s=2$
    \begin{align*}
      \norm{\tilde W-I_2\tilde W}{L^2}
        &\lesssim (h+N^{-1})^2\norm{(\tilde W)''}{L^2}
         \lesssim \eps^{1/2}(h+N^{-1})^2
    \end{align*}
    and for $q=2$ we use $s=2$ and $s=3$ together
    \begin{align*}
      \norm{\tilde W-I_2\tilde W}{L^2}
        &\lesssim ((h+N^{-1})^2\norm{(\tilde W)''}{L^2})^{1/2}
                  ((h+N^{-1})^3\norm{(\tilde W)'''}{L^2})^{1/2}
         \lesssim (h+N^{-1})^{5/2}.
    \end{align*}
    For $q>2$ the same trick can be applied outside the inner layer region,
    \begin{align*}
      \norm{\tilde W-I_2\tilde W}{L^2((0,1-\mu)\cap(1+\mu,2))}
        &\lesssim (N^{-2}\norm{(\tilde W)''}{L^2((0,1-\mu)\cap(1+\mu,2))})^{1/2}\cdot\\&\hspace*{0.6cm}
                  (N^{-3}\norm{(\tilde W)'''}{L^2((0,1-\mu)\cap(1+\mu,2))})^{1/2}\\
        &\lesssim N^{-5/2}\e^{-\eps^{-1/q}}
         \lesssim N^{-(q+1/2)},
    \end{align*}
    while inside we estimate
    \begin{align*}
      \norm{\tilde W-I_2\tilde W}{L^2(1-\mu,1+\mu)}
        &\lesssim \tilde h^{q+1}\norm{\tilde W^{(q+1)}}{L^2(1-\mu,1+\mu)}
         \lesssim N^{-(q+1)}\eps^{q+1-\frac{5}{2}}\eps^{3/2-q}
         \lesssim N^{-(q+1)}.
    \end{align*}
    For $W$ it follows by the same steps
    \[
      \norm{W-I_1 W}{L^2}
        \lesssim (h+N^{-1})^{q+1},
    \]
    using \eqref{eq:inter:local} with $s=3$ and $s=4$ where appropriate. 
    The estimation of the derivative follows also the same steps and we obtain
    \[
      \norm{(W-I_1 W)'}{L^2}
        \lesssim (h+N^{-1})^{q}.
    \]
    Combining all these estimates completes the proof.
  \end{proof}
  The convergence proof is now a straightforward consequence. 
  \begin{thm}\label{thm:conv:weakeq}
    Under the conditions $\sigma\geq q+1$ and $\e^{-\eps^{-1/q}}\leq(h+N^{-1})^{q-2}$ it holds for 
    the exact solution $(u,w)$ of \eqref{eq:varprob} and the discrete solution $(u_h,w_h)$ of \eqref{eq:discr}
    the supercloseness estimate
    \[
      \tnorm{(I_1 u-u_h,I_2 w-w_h)}\lesssim (h+N^{-1}\max|\psi'|)^{q+1/2}
    \]
    and the convergence estimate
    \[
      \tnorm{(u-u_h,w-w_h)}\lesssim (h+N^{-1}\max|\psi'|)^q.
    \]
  \end{thm}
  \begin{rem}
    The condition $\e^{-\eps^{-1/q}}\leq (h+N^{-1})^{q-2}$ limits the applicability
    of the mesh for higher values of $q$ and small $\eps$, see also \cite{BFrLR22b}, where 
    the condition was similar. For $q\leq 2$ we have no restriction, and for $q=3$ the 
    condition is satisfied for reasonable choices of $N$ and $\eps$.
    Nevertheless, for $q>2$ in the next subsection we provide another coarse mesh 
    which does not restrict the choices of $N$ and $\eps$.
  \end{rem}

  \subsection{Weak S-type mesh}
  In this section we use an S-type mesh for the inner layer in the case of $q>2$, but we modify its transition 
  point with respect to the weak layer. Let 
  \[
    \nu:=\min\left\{\frac{\sigma}{\beta }\eps^{1-\frac{5}{2(q+1)}}\ln N,\frac{1}{4}\right\}\geq\lambda
  \]
  and the mesh is defined as in \eqref{eq:Smesh} with the obvious modification of using $\nu$ instead of $\lambda$
  near $x=1$.
  One consequence is the following bound on the mesh width $h_i$ in the inner layer region
  \[
    h_i\leq \frac{\sigma}{\beta }\eps^{\frac{2q-3}{2(q+1)}}N^{-1}\max|\psi'|\exp\left({\frac{\beta x}{\sigma}\eps^{-\frac{2q-3}{2(q+1)}}}\right)\lesssim\tilde h
  \]
  for $x\in(x_{i-1},x_i)$ and the maximum mesh width $h\lesssim\tilde h\lesssim\eps^{1-\frac{5}{2(q+1)}}$.
  
  As in the previous section, we will first look at the interpolation error.
  \begin{lem}\label{lem:inter:weakStype}
    Let $\sigma\geq q+1$. Then we have for the interpolation operator $I$ on the weak S-type mesh
    \begin{align*}
      \norm{u-I_1 u}{L^2}&\lesssim (\tilde h+N^{-1}\max|\psi'|)^{q+1},\\
      \norm{w-I_2 w}{L^2}&\lesssim (\tilde h+N^{-1}\max|\psi'|)^{q+1},\\
      \norm{(u-I_1 u)'}{L^2}&\lesssim (\tilde h+N^{-1}\max|\psi'|)^q.
    \end{align*}
  \end{lem}
  \begin{proof}
    As for Lemma~\ref{lem:inter:weakeq}, we only need to estimate the interpolation errors for the inner
    layers. The estimates for the remaining terms only need to be modified by using $h\lesssim\tilde h$ due to 
    the new maximum mesh width in the inner layer region.
    
    For $\tilde W$ we obtain in the layer region with \eqref{eq:inter:local}
    \begin{align*}
      \norm{\tilde W-I_2\tilde W}{L^2(1-\nu,1)}^2
        &\lesssim \sum_{\tau_i\in\Omega_h([1-\nu,1])}\norm{h_i^{q+1}\tilde W_1^{(q+1)}}{L^2(\tau_i)}^2\\
        &\lesssim \eps^{-1}(N^{-1}\max|\psi'|)^{2(q+1)}\bignorm{\exp\left(\left(\frac{q+1}{\sigma\eps^{\frac{2q-3}{2(q+1)}}}-\frac{1}{\eps}\right)\beta x\right)}{L^2(1-\nu,1)}^2\\
        &\lesssim (N^{-1}\max|\psi'|)^{2(q+1)}
    \end{align*}
    due to $\sigma\geq q+1$ and $\eps^{\frac{2q-3}{2(q+1)}}=\eps^{1-\frac{5}{2(q+1)}}>\eps$. In the other layer region the same result holds. 
    Outside the layer region we have by $\nu>\lambda$ and the $L^\infty$-stability of the interpolation operator
    \[
      \norm{\tilde W-I_2\tilde W}{L^2((0,1-\nu)\cap(1+\nu,2))}
        \lesssim \norm{\tilde W}{L^\infty((0,1-\nu)\cap(1+\nu,2)))}
        \lesssim |\tilde W(1-\nu)|
        \lesssim |\tilde W(1-\lambda)|
        \lesssim \eps^2 N^{-\sigma}.
    \]
    For $W$ we obtain analogously
    \[
      \norm{W-I_1 W}{L^2}\lesssim \eps^2 (N^{-1}\max|\psi'|)^{q+1}.
    \]
    Thus, the first two estimates are proven. For the derivative we look at $(0,\lambda)$, $(\lambda,1-\nu)$ and $(1-\nu,1)$ separately:
    \begin{align*}
      \norm{(W_1-I_1 W_1)'}{L^2(0,\lambda)}
        &\lesssim h^q\norm{W_1^{(q+1)}}{L^2(0,\lambda)}
         \lesssim \eps^{q}|W_1^{(q+1)}(1-\lambda)|
         \lesssim \eps^2 N^{-\sigma},\\
      \norm{(W_1-I_1 W_1)'}{L^2(\lambda,1-\nu)}
        &\lesssim \norm{W_1'}{L^2(\lambda,1-\nu)}+N\norm{W_1}{L^\infty(\lambda,1-\nu)}
         \lesssim \eps^{5/2}N^{-\sigma}+\eps^3N^{-(\sigma-1)},\\
      \norm{(W_1-I_1 W_1)'}{L^2(1-\nu,1)}
        &\lesssim \eps^{\frac{2q-3}{2(q+1)}q}(N^{-1}\max|\psi'|)^q \eps^{1-q}\bignorm{\exp\left( \left( \frac{q}{\sigma\eps^{\frac{2q-3}{2(q+1)}}}-\frac{1}{\eps} \right) \right)}{L^2(1-\nu,1)}\\
        &\lesssim \eps^{\frac{5}{2(q+1)}}(N^{-1}\max|\psi'|)^q.
    \end{align*}
    On the other parts of the domain we get the same bounds.
    When estimating $(S-I_1 S)'$ we can proceed as before, but for $(E-I_1 E)'$ we have to be careful in the middle layer region.
    Direct estimation using \eqref{eq:inter:local} with $s=2$ works for $q>1$, but using the $W_1^\infty$-stability of $I_1$
    works for all $q$:
    \begin{align*}
      \norm{(E_1-I_1 E_1)'}{L^2(1-\nu,1)}
        &\lesssim \norm{E_1'}{L^2(1-\nu,1)}+\nu^{1/2}\norm{E_1'}{L^\infty(1-\nu,1)}
         \lesssim 
                  \eps^{m-1}N^{-\sigma}
         \lesssim N^{-\sigma}
    \end{align*}
    due to $m\geq 1$. Combining all estimates finishes the proof.
  \end{proof}
  With the interpolation error result we obtain virtually the same result as in Theorem~\ref{thm:conv:Smesh}.
  \begin{thm}[Supercloseness and convergence]\label{thm:conv:weakSmesh}
    For the solutions $(u,w)\in\U$ of \eqref{eq:varprob} and $(u_h,w_h)\in\V$ of \eqref{eq:discr} it holds under the assumptions
    $u,w\in H^{q+1}(\Omega)$ and $\sigma\geq q+1$, the supercloseness result
    \[
      \tnorm{(I_1u-u_h,I_2w-w_h)}\lesssim(\tilde h+N^{-1}\max|\psi'|)^{q+1}
    \]
    and the convergence result
    \[
      \tnorm{(u-u_h,w-w_h)}\lesssim(\tilde h+N^{-1}\max|\psi'|)^q.
    \]
  \end{thm}
  \begin{rem}
    Let us compare the two coarser meshes for the inner layers and the results on them.
    \begin{itemize}
      \item The weak equidistant mesh is structurally simpler.
      \item In terms of size with respect to $\eps$, the transition points $\mu$ and
            and $\nu$ are of the same size.
      \item On both meshes convergence of order $q$ is proved.
      \item On the weak S-type mesh we have supercloseness of order $q+1$, while on the weak equidistant mesh we can only prove the order $q+1/2$.
      \item The maximum mesh size $\tilde h$ on weak S-type meshes is for optimal meshes, such as the Bakhvalov S-mesh, of order
            $\ord{\eps^{1-\frac{5}{2(q+1)}}}$. Then we have optimal convergence for $\eps\lesssim N^{-\left( 1+\frac{4}{2q-3} \right)}$. 
    \end{itemize}
    As a result, a practical guide for the inner layer is to use
    \begin{itemize}
      \item for arbitrary $\eps$ and $q\leq 2$ a piecewise equidistant mesh,
      \item for small $\eps$ and $q=3$ a weak equidistant mesh and
      \item otherwise a weak S-type or a classical S-type mesh.
    \end{itemize}
    Using a classical Shishkin mesh for the inner layer in the case of $q\geq 3$ will also yield optimal convergence results (without a logarithmic factor) under the assumption that
    \[
      \eps\lesssim (\ln N)^{-\frac{2q}{5}}
    \]
    which is slightly stronger than
    \[
      \lambda\leq \frac{1}{4}
      \rarrow
      \eps\lesssim (\ln N)^{-1}
    \]
    but much weaker than $\eps\lesssim N^{-1}$.
  \end{rem}
  
\section{Numerical experiments}\label{sec:numerics}
  We consider as first example in $\Omega$
  \begin{gather}\label{eq:ex}
    \eps^2 u^{(4)}(x)-u''(x)+2u(x)+u(x-1)=5
  \end{gather}
  with $u(0)=u'(0)=u(2)=u'(2)=0$, thus $m=1$, and $\Phi(x)=0$ for $x\in(-1,0)$. Figure~\ref{fig:plot}
  \begin{figure}[htb]
    \begin{center}
      \includegraphics{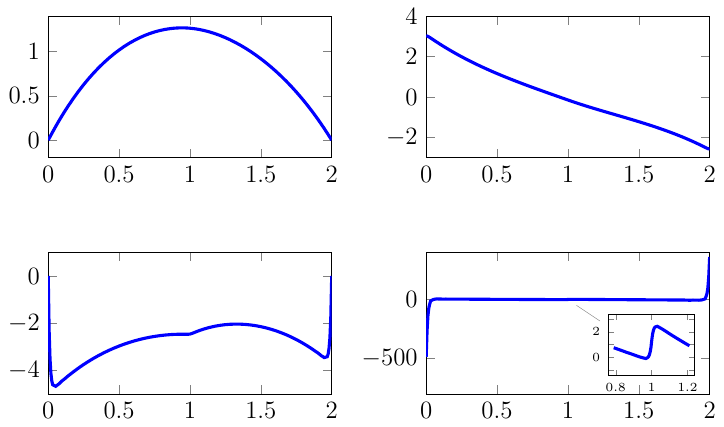}
    \end{center}
    \caption{Solution $u$ and $u'$, $u''$, $u'''$ (left to right, top to bottom), for example \eqref{eq:ex} with $\eps=10^{-2}$.\label{fig:plot}}
  \end{figure}
  shows the solution $u$ and its first three derivatives. It is easy to see, that $u'$ has two boundary layers and $u'''$ an inner layer, visible in the zoom.

  Our second example is almost the same, except for the boundary conditions. We consider in $\Omega$
  \begin{gather}\label{eq:ex2}
    \eps^2 u^{(4)}(x)-u''(x)+2u(x)+u(x-1)=5
  \end{gather}
  with $u(0)=u''(0)=u(2)=u''(2)=0$, thus $m=2$, and $\Phi(x)=0$ for $x\in(-1,0)$. Figure~\ref{fig:plot2}
  \begin{figure}[htb]
    \begin{center}
      \includegraphics{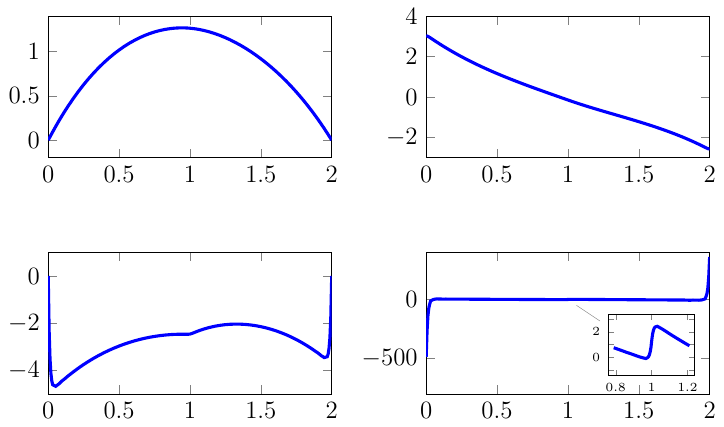}
    \end{center}
    \caption{Solution $u$ and $u'$, $u''$, $u'''$ (left to right, top to bottom), for example \eqref{eq:ex2} with $\eps=10^{-2}$.\label{fig:plot2}}
  \end{figure}
  shows the solution $u$ and its first three derivatives. We see, that now $u''$ has two boundary layers and $u'''$ still has an inner layer, visible in the zoom.

  For our simulations we use the Matlab finite element suite $\mathbb{SOFE}$\footnote{\texttt{github.com/SOFE-Developers/SOFE}} and all computations are 
  done in double precision. Table~\ref{tab:ex1}
  \begin{table}[tb]
    \caption{Errors for example~\eqref{eq:ex} and $\eps=10^{-4}$ on a Bakhvalov S-type mesh for all layers.\label{tab:ex1}}
    \begin{center}
      \begin{tabular}{rr|rcrcrc}
        $q$ & $N$ & \multicolumn{2}{c}{$\tnorm{(u-u_h,w-w_h)}$}
                  & \multicolumn{2}{c}{$\norm{u-u_h}{L^2}$}
                  & \multicolumn{2}{c}{$\norm{w-w_h}{L^2}$}\\
        \hline
        \multirow{3}{*}{1} 
                           &  64 & 5.27e-02 & 1.00 & 1.10e-03 & 2.00 & 3.19e-04 & 1.90\\
                           & 128 & 2.64e-02 & 1.00 & 2.75e-04 & 2.00 & 8.54e-05 & 1.92\\
                           & 256 & 1.32e-02 &      & 6.87e-05 &      & 2.26e-05 &     \\
        \hline
        \multirow{3}{*}{2} 
                           &  64 & 4.51e-04 & 2.00 & 5.60e-06 & 3.00 & 4.18e-05 & 2.86\\
                           & 128 & 1.12e-04 & 2.00 & 7.00e-07 & 3.00 & 5.77e-06 & 2.89\\
                           & 256 & 2.81e-05 &      & 8.76e-08 &      & 7.78e-07 &     \\
        \hline
        \multirow{3}{*}{3} 
                           &  64 & 1.66e-05 & 3.00 & 5.28e-08 & 2.93 & 3.18e-06 & 3.84\\
                           & 128 & 2.08e-06 & 2.99 & 6.92e-09 & 0.23 & 2.22e-07 & 3.88\\
                           & 256 & 2.61e-07 &      & 5.90e-09 &      & 1.50e-08 &     \\
        \hline
        \multirow{5}{*}{4} &  16 & 3.41e-04 & 4.02 & 1.36e-07 & 4.18 & 1.76e-04 & 4.61\\
                           &  32 & 2.10e-05 & 4.00 & 7.51e-09 & 0.26 & 7.22e-06 & 4.76\\
                           &  64 & 1.32e-06 & 3.97 & 6.25e-09 & 0.09 & 2.66e-07 & 4.84\\
                           & 128 & 8.40e-08 & 2.74 & 5.86e-09 &-0.04 & 9.31e-09 & 4.77\\
                           & 256 & 1.26e-08 &      & 6.01e-09 &      & 3.41e-10 &     \\
      \end{tabular}      
    \end{center}
  \end{table}
  shows the results on a mesh consisting of Bakhvalov S-type meshes for all layers. We can confirm the 
  predictions made in Theorem~\ref{thm:conv:Smesh} of convergence in the $\tnorm{\cdot}$-norm of order $q$. 
  Moreover, Remark~\ref{rem:super} mentions as consequence of a supercloseness result an optimal error 
  convergence in $L^2$ of order $q+1$. This can also be observed in Table~\ref{tab:ex1}. 
  Note that for higher values of $N$ and higher polynomial degrees we have a collapse of the computational
  accuracy. This is a well-known phenomenon in singularly perturbed problems, and we cannot achieve 
  machine precision. To get around this, we would have to use quadruple or arbitrary 
  precision.
  
  In Table~\ref{tab:ex1:compare}
  \begin{table}[tb]
    \caption{Errors for example~\eqref{eq:ex}, $q=2$ and $\eps=10^{-4}$ on different meshes.\label{tab:ex1:compare}}
    \begin{center}
      \begin{tabular}{r|rcrcrcrc}
        $N$ & \multicolumn{2}{c}{BS-BS}
            & \multicolumn{2}{c}{BS-Shishkin}
            & \multicolumn{2}{c}{BS-weakeq}
            & \multicolumn{2}{c}{BS-weakShishkin}\\
        \hline
         64 & 4.51e-04 & 2.00 & 4.51e-04 & 2.00 & 4.43e-04 & 2.00 & 3.60e-04 & 2.06\\
        128 & 1.12e-04 & 2.00 & 1.12e-04 & 2.00 & 1.11e-04 & 2.00 & 8.66e-05 & 2.05\\
        256 & 2.81e-05 &      & 2.81e-05 &      & 2.77e-05 &      & 2.09e-05 &     \\  
      \end{tabular}
    \end{center}
  \end{table}
  we compare the results on different meshes. We restrict the polynomial degree to $q=2$ and look only at the $\tnorm{\cdot}$-norm.
  If we use a Bakhvalov S-mesh for the boundary layer it does not matter whether we use a Bakhvalov S-mesh or a Shishkin mesh
  for the inner layer. The results are the same. However, for $q=2$ we can also use an equidistant mesh for the inner layer
  and get slightly better results. Even better results can be obtained by using a weak S-type mesh, such as a weak Shishkin
  mesh. But overall they all perform similarly well and we observe convergence of order $q=2$.
  
  For example \eqref{eq:ex2} we can do the same experiments and observe similar results. We only show the results corresponding to Table~\ref{tab:ex1:compare}
  in Table~\ref{tab:ex2:compare}.
  \begin{table}[htb]
    \caption{Errors for example~\eqref{eq:ex2}, $q=2$ and $\eps=10^{-4}$ on different meshes.\label{tab:ex2:compare}}
    \begin{center}
      \begin{tabular}{r|rcrcrcrc}
        $N$ & \multicolumn{2}{c}{BS-BS}
            & \multicolumn{2}{c}{BS-weakeq}
            & \multicolumn{2}{c}{BS-weakShishkin}\\
        \hline
         64 & 4.12e-04 & 2.00 & 4.04e-04 & 2.00 & 3.10e-04 & 2.07\\
        128 & 1.03e-04 & 2.00 & 1.01e-04 & 2.00 & 7.38e-05 & 2.07\\
        256 & 2.57e-05 &      & 2.52e-05 &      & 1.75e-05 &     \\  
      \end{tabular}
    \end{center}
  \end{table}
%

  \appendix
  \section{Proof of Lemma~\ref{lem:stability}}
  The proof of Lemma~\ref{lem:stability} relies on properties of an associated Green's function
  that we will derive now.
  
  Let $G$ be the Green's function defined for all $t\in(0,1)$ by
  \begin{subequations}\label{eq:Green}
  \begin{align}
    \eps^2 \pt_x^4 G(x,t)-b \pt_x^2G(x,t)+cG(x,t)&=\delta(x-t),\,x\in\Omega=(0,1),\label{eq:Green:eq}\\
    G(0,t)=G(2,t)&=0,\\
    G_x(0,t)=G_x(2,t)&=0.\label{eq:Green:BC2}
  \end{align}
  \end{subequations}
  
  The characteristic solutions of \eqref{eq:Green:eq} are given by 
  \begin{align*}
    y_1(x)&=\e^{-\mu_1 x},\,
    y_2(x)=\e^{-\mu_2 x},\,
    y_3(x)=\e^{-\mu_1 (1-x)},\,
    y_4(x)=\e^{-\mu_2 (1-x)},
  \end{align*}
  where 
  \[
    \mu_1=\sqrt{\frac{b+\sqrt{b^2-4\eps^2 c}}{2\eps^2}}\sim \sqrt{b}\eps^{-1},\,
    \mu_2=\sqrt{\frac{b+\sqrt{b^2-4\eps^2 c}}{2\eps^2}}\sim \sqrt{\frac{c}{b}}.
  \]

  We find a representation of $G$, see e.g. \cite[Sec. 3.3]{Naimark67}, using
  \[
    G(x,t)=\begin{cases}
             \sum\limits_{i=1}^4 c_i^\ell(t)y_i(x),&x<t,\\
             \sum\limits_{i=1}^4 c_i^r(t)y_i(x),&x\geq t,
           \end{cases}
  \]
  where the unknowns $c_i^\ell$ and $c_i^r$ for each $t$ can be found by solving the system
  \begin{align*}
    G(0,t)&=0,&   
    G(1,t)&=0,&
    G_x(0,t)&=0,&
    G_x(1,t)&=0,\\
    \jump{G}(t,t)&=0,&
    \jump{G_x}(t,t)&=0,&
    \jump{G_{xx}}(t,t)&=0,&
    \jump{G_{xxx}}(t,t)&=\eps^{-2}.
  \end{align*}
  These linearly independent conditions determine $G(x,t)$ uniquely for each $t\in(0,1)$.

  Now we want to represent $u_1$ and $u_2$ using $G$. For this purpose we define by $\Phi_i$,\,$i\in\{1,\dots,4\}$
  the Hermite basis of $\PS_3(0,1)$ associated with the evaluation of function and first derivative values
  and interpolate the boundary conditions in \eqref{eq:problem:BC2} by
  \begin{align*}
    \phi_1(x)&=\alpha_1\Phi_1(x)+\alpha_2\Phi_2(x)+\beta_1\Phi_3(x)+\beta_2\Phi_4(x),\,x\in(0,1),\\
    \phi_2(x)&=(\alpha_2+\delta_1)\Phi_1(x-1)+\alpha_3\Phi_2(x-1)+(\beta_2+\delta_2)\Phi_3(x-1)+\beta_3\Phi_4(x-1),\,x\in(1,2).
  \end{align*}
  Then $v_i:=u_i-\phi_i$ satisfies homogeneous boundary conditions and we can reformulate
  the continuity conditions \eqref{eq:cont} for finding $\alpha_2$ and $\beta_2$ as
  \begin{align*}
    \jump{v''}(1)&
                  =\phi''_2(1)-\phi''_1(1)
                  =6(\alpha_3-\alpha_1-\delta_1)-2(\beta_1+4\beta_2+\beta_3+2\delta_2)=:g_1-8\beta_2,\\
    \jump{v'''}(1)&
                   =\phi'''_2(1)-\phi'''_1(1)
                   =-12(\alpha_1-2\alpha_2+\alpha_3-\delta_1)-6(\beta_1-\beta_3-\delta_2)=:g_2+24\alpha_2.
  \end{align*}
  In order to evaluate the remaining jump terms, we represent $v_1$ and $v_2$ using the Green's function
  \begin{subequations}\label{eq:represent}
  \begin{align}
    v_1(x)&=\int_0^1 G(x,t)\left(f(t)+b\phi_1''(t)-c\phi_1(t)\right)\dt,\\
    v_2(x)&=\int_0^1 G(x-1,t)\left(f(t+1)+b\phi_2''(t+1)-c\phi_2(t+1)+d(\phi_1(t)-v_1(t))\right)\dt\notag\\
          &=\int_0^1 G(x-1,t)\bigg(f(t+1)+b\phi_2''(t+1)-c\phi_2(t+1)+d\phi_1(t)\notag\\&\hspace{4cm}
                                   -d\int_0^1 G(t,s)\left(f(s)+b\phi_1''(s)-c\phi_1(s)\right)\ds\bigg)\dt.
  \end{align}
  \end{subequations}
  To be more precise, we have
  \begin{align*}
    v_1(x)&=g_3(x)+\alpha_2 F_1(x)+\beta_2 F_2(x),\\
    v_2(x)&=g_4(x)+\alpha_2 F_3(x)+\beta_2 F_4(x),
  \end{align*}
  where
  \begin{align*}
    g_3(x)&=\int_0^1 G(x,t)(f(t)+b(\alpha_1\Phi_1''(t)+\beta_1\Phi_3''(t))-c(\alpha_1\Phi_1(x)+\beta_1\Phi_3(t)))\dt,\\
    g_4(x)&=\int_0^1 G(x-1,t)\Bigg(f(t+1)+b(\delta_1\Phi_1''(t)+\alpha_3\Phi_2''(t)+\delta_2\Phi_3''(t)+\beta_3\Phi_4''(t))\\&\hspace{5cm}
                                    -c(\delta_1\Phi_1(t)+\alpha_3\Phi_2(t)+\delta_2\Phi_3(t)+\beta_3\Phi_4(t))\\&\hspace{5cm}
                                    +d(\alpha_1\Phi_1(t)+\beta_1\Phi_3(t))
                                    -dg_3(t)\Bigg)\dt,\\
    F_1(x)&=\int_0^1 G(x,t)(b\Phi_2''(t)-c\Phi_2(t))\dt,\\
    F_2(x)&=\int_0^1 G(x,t)(b\Phi_4''(t)-c\Phi_4(t))\dt,\\
    F_3(x)&=\int_0^1 G(x-1,t)\left(b\Phi_1''(t)-c\Phi_1(t)+d\Phi_2(t)-dF_1(t)\right)\dt,\\
    F_4(x)&=\int_0^1 G(x-1,t)\left(b\Phi_3''(t)-c\Phi_3(t)+d\Phi_4(t)-dF_2(t)\right)\dt.
  \end{align*}
  The system we need to solve reads after rescaling
  \begin{gather*}
    \pmtrx{[c]\eps(F_3''(1)-F_1''(1)) & \eps(F_4''(1)-F_2''(1)+8)\\
           \eps^2(F_3'''(1)-F_1'''(1)-24) & \eps^2(F_4'''(1)-F_2'''(1))}
    \pmtrx{\alpha_2\\\beta_2}
    =\pmtrx{[c]\eps(g_1-g_4''(1)+g_3''(1))\\\eps^2(g_2-g_4'''(1)+g_3'''(1))}.
  \end{gather*}
  In the following we show, that for small $\eps$ the coefficient matrix, let us call it $\mA$, is regular and its inverse is $\ord{1}$.
  To do this, we look at the integrals associated with $G$ and compute their leading terms in an $\eps$-expansion using MAPLE.
  We obtain
  \begin{align*}
    \int_0^1 G_{xx}(1,t)\dt   &= \frac{\e-1}{\e+1}\eps^{-1}+\ord{1},&
    \int_0^1 G_{xx}(1,t)t\dt  &= \frac{2}{\e^2-1}\eps^{-1}+\ord{1},\\
    \int_0^1 G_{xx}(1,t)t^2\dt&= \frac{\e^2-4\e+5}{\e^2-1}\eps^{-1}+\ord{1},&
    \int_0^1 G_{xx}(1,t)t^3\dt&= \frac{16-2\e^2}{\e^2-1}\eps^{-1}+\ord{1},\\
    \int_0^1 G_{xx}(0,t)\dt   &= \frac{\e-1}{\e+1}\eps^{-1}+\ord{1},&
    \int_0^1 G_{xx}(0,t)t\dt  &= \frac{e^2-2e-1}{\e^2-1}\eps^{-1}+\ord{1},\\
    \int_0^1 G_{xx}(0,t)t^2\dt&= \frac{2\e^2-6\e+2}{\e^2-1}\eps^{-1}+\ord{1},&
    \int_0^1 G_{xx}(0,t)t^3\dt&= \frac{6\e^2-14\e-6}{\e^2-1}\eps^{-1}+\ord{1}.
  \end{align*}
  Furthermore, it holds for $k\in\{0,\dots,3\}$
  \begin{align*}
    \int_0^1 G_{xxx}(1,t)t^k\dt&=\eps^{-1}\left(\int_0^1 G_{xx}(1,t)t^k\dt+\ord{1}\right),\\
    \int_0^1 G_{xxx}(0,t)t^k\dt&=-\eps^{-1}\left(\int_0^1 G_{xx}(0,t)t^k\dt+\ord{1}\right).
  \end{align*}
  In addition, we also need estimates for the double-integrals
  \begin{align*}
    \int_0^1 G_{xx}(0,t)\int_0^1 G(t,s)\ds\dt    &= \frac{\e^2-2\e-1}{2(1+\e)^2}\eps^{-1}+\ord{1},\\
    \int_0^1 G_{xx}(0,t)\int_0^1 G(t,s)s\ds\dt   &= \frac{\e^4-2\e^3-2\e^2+1}{(\e^2-1)^2}\eps^{-1}+\ord{1},\\
    \int_0^1 G_{xx}(0,t)\int_0^1 G(t,s)s^2\ds\dt &= \frac{3\e^4-10\e^3+4\e^2+4\e-3}{(\e^2-1)^2}\eps^{-1}+\ord{1},\\
    \int_0^1 G_{xx}(0,t)\int_0^1 G(t,s)s^3\ds\dt &= \frac{12\e^4-26\e^3-24\e^2+12\e+12}{(\e^2-1)^2}\eps^{-1}+\ord{1},\\
    \int_0^1 G_{xxx}(0,t)\int_0^1G(t,s)s^k\ds\dt &=-\eps^{-1}\left(\int_0^1 G_{xx}(0,t)\int_0^1G(t,s)s^k\ds\dt+\ord{1}\right),
  \end{align*}
  where $k\in\{0,\dots,3\}$.
  
  With these estimates we obtain for the entries in $\mA$ of the scaled system 
  \begin{align*}
    \eps(F_3''(1)-F_1''(1))   &= 12\frac{3-\e}{\e-1}b
                                +6\frac{3\e^4-8\e^3+1}{(\e^2-1)^2}bd\\&\hspace{0.6cm}
                                +\frac{17\e^2-26\e-39}{\e^2-1}c
                                -\frac{15\e^4-22\e^3-60\e^2+12\e+33}{(\e^2-1)^2}cd\\&\hspace{0.6cm}
                                -2\frac{3\e^2-5\e-9}{\e^2-1}d+\ord{\eps},
  \end{align*}
  \begin{align*}
    \eps(F_4''(1)-F_2''(1)+8) &= 6\frac{\e-3}{\e-1}b
                                -3\frac{3\e^4-8\e^3+1}{(\e^2-1)^2}bd\\&\hspace{0.6cm}
                                -\frac{7\e^2-10\e-21}{\e^2-1}c
                                +\frac{9\e^4-16\e^3-28\e^2+8\e+15}{(\e^2-1)^2}cd\\&\hspace{0.6cm}
                                +4\frac{\e^2-2\e-2}{\e^2-1}d+\ord{\eps},\\
    \eps^2(F_3'''(1)-F_1'''(1)-24) &= 36\frac{\e-1}{\e+1}b
                                     -6\frac{3\e^4-8\e^3+1}{(\e^2-1)^2}bd\\&\hspace{0.6cm}
                                     -\frac{3\e-5}{\e-1}c
                                     +\frac{15\e^4-22\e^3-60\e^2+12\e+33}{\e^2-1}cd\\&\hspace{0.6cm}
                                     +2\frac{3\e^2-5\e-9}{\e^2-1}d+\ord{\eps},
  \end{align*}
  \begin{align*}
    \eps^2(F_4'''(1)-F_2'''(1))    &=-18\frac{\e-1}{\e+1}b
                                     +3\frac{3\e^4-8\e^3+1}{\e^2-1}bd\\&\hspace{0.6cm}
                                     +\frac{\e-1}{\e+1}c
                                     -\frac{9\e^4-16\e^3-28\e^2+8\e+15}{(\e^2-1)^2}cd\\&\hspace{0.6cm}                                     
                                     -4\frac{\e^2-2\e-2}{\e^2-1}d+\ord{\eps}.
  \end{align*}
  Furthermore, for its determinant we have
  \begin{align*}
    \det(\mA)
      &= -4\frac{\e^4 + 4\e^3 - 27\e^2 + 10\e +36}{(\e^2-1)^2} c^2
         -24\frac{\e^4-5\e^3+10\e^2-11\e+3}{(\e^2-1)^2}bd\\&\hspace{0.6cm}
         -24\frac{2\e^4-9\e^3+17\e^2-11\e-3}{(\e^2 -1)^2}bc
        -4\frac{5\e^2-25\e+31}{1-\e^(2)}cd\\&\hspace{0.6cm}
        -4\frac{9\e^4-47\e^3+59\e^2+26\e-54}{(\e^2-1)^2}c^2d
        -6\frac{3\e^4-12\e^3+22\e^2-40\e+23}{(\e^2-1)^2}bcd+\ord{\eps}.
  \end{align*}
  We therefore conclude, that $\mA$ is regular and $\norm{\mA^{-1}}{\infty}\lesssim 1$. Thus
  \[
    \bignorm{\pmtrx{\alpha_2\\\beta_2}}{\infty}\lesssim\bignorm{\pmtrx{[c]\eps(g_1-g_4''(1)+g_3''(1))\\\eps^2(g_2-g_4'''(1)+g_3'''(1))}}{\infty},
  \]
  and for the final stability result we have to estimate the right hand side. For that let us collect
  the absolute values of the data:
  \[
    \kappa:=|\alpha_1|+|\alpha_3|+|\delta_1|+|\delta_2|+|\beta_1|+|\beta_3|.
  \]
  Then the right hand side can be estimated by
  \begin{align*}
    \eps|g_1|        &\lesssim \eps \kappa,\quad
    \eps^2|g_2|       \lesssim \eps^2\kappa,\\
    \norm{g_3}{L^\infty(0,1)} &\lesssim \sup_{x\in(0,1)}\int_0^1 G(x,t)\dt(\norm{f}{L^\infty(0,1)}+|\alpha_1|+|\beta_1|),
  \end{align*}
  \begin{align*}
    \eps|g_3''(1)|   &\lesssim \eps\int_0^1 |G_{xx}(1,t)|\dt(\norm{f}{L^\infty(0,1)}+|\alpha_1|+|\beta_1|),\\
    \eps^2|g_3'''(1)|&\lesssim \eps^2\int_0^1 |G_{xxx}(1,t)|\dt(\norm{f}{L^\infty(0,1)}+|\alpha_1|+|\beta_1|),\\
    \eps|g_4''(1)|&\lesssim\eps\int_0^1 |G_{xx}(0,t)|\dt\Bigg(\norm{f}{L^\infty(1,2)}+\kappa+\norm{g_3}{L^\infty(0,1)}\Bigg),\\
    \eps^2|g_4'''(1)|&\lesssim\eps^2\int_0^1 |G_{xxx}(0,t)|\dt\Bigg(\norm{f}{L^\infty(1,2)}+\kappa+\norm{g_3}{L^\infty(0,1)}\Bigg),\\
  \end{align*}
  The final ingredients are the following estimates on the Green's function
  \begin{align*}
    0\leq\int_0^1 G(x,t)\dt
     \leq \int_0^1 G\left(\frac{1}{2},t\right)\dt
     &=\frac{(\e^{1/2}-1)^2}{\e+1}+\ord{\eps},\\
    \int_0^1|G_{xx}(0,t)|\dt+\int_0^1|G_{xx}(1,t)|\dt&\lesssim \eps^{-1},\\
    \int_0^1|G_{xxx}(0,t)|\dt+\int_0^1|G_{xxx}(1,t)|\dt&\lesssim \eps^{-2}.
  \end{align*}
  So for the unknown values at $x=1$ we get
  \[
    |\alpha_2|+|\beta_2|\lesssim \norm{f}{L^\infty(0,2)}+\kappa.
  \]
  Now using $u_i=v_i+\phi_i$ and the representation formulae \eqref{eq:represent}
  we finally obtain
  \[
    \norm{u}{L^\infty(0,2)}\lesssim \norm{f}{L^\infty(0,2)}+\kappa.
  \]
  
  The above result is for the case $m=1$ -- first derivatives as boundary data. For $m=2$ we can follow the same procedure.
  Now \eqref{eq:problem:BC2} is replaced by
  \begin{align*}
    u_1''(0)=\beta_1,\,
    u_1''(1)&=\beta_2=
    u_2''(1),\,
    u_2''(2)=\beta_3,
  \end{align*}
  and $\alpha_2,\,\beta_2$ are chosen such that
  \[
    \jump{u'}(1)=\delta_2,\,
    \jump{u'''}(1)=0.
  \]

  To define the associated Green's function, we change the conditions \eqref{eq:Green:BC2} to
  \begin{align*}
    G_{xx}(0,t)&=0,\,
    G_{xx}(1,t)=0
  \end{align*}
  and find a similar representation of $G$.
  In order to represent $u$ by $G$, we again write $v_i=u_i-\phi_i$, where
  \begin{align*}
    \phi_1(x)&=\alpha_1\Psi_1(x)+\alpha_2\Psi_2(x)+\beta_1\Psi_3(x)+\beta_2\Psi_4(x),\,x\in(0,1),\\
    \phi_2(x)&=(\alpha_2+\delta_1)\Psi_1(x-1)+\alpha_3\Psi_2(x-1)+\beta_2\Psi_3(x-1)+\beta_3\Psi_4(x-1),\,x\in(1,2)
  \end{align*}
  and $\{\Psi_k\}$ are the Hermite basis of $\PS_3$ associated with the evaluation of function and second derivative values.
  The rest of the steps are similar and so is the result.

\end{document}